\documentclass{article}%
\usepackage{amsmath}
\usepackage{amsfonts}%
\usepackage{amssymb}%
\usepackage{graphicx}
\usepackage{hyperref}
\hypersetup{colorlinks=true,linkcolor=blue,citecolor=blue}
\providecommand{\U}[1]{\protect\rule{.1in}{.1in}}
\newtheorem{theorem}{Theorem}

\newtheorem{corollary}{Corollary}

\newtheorem{lemma}{Lemma}

\newtheorem{remark}{Remark}

\newenvironment{proof}[1][Proof]{\noindent\textbf{#1.} }{\ \rule{0.5em}{0.5em}}
\begin{document}

\title{On $C$-Pareto dominance in decomposably $C$-antichain-convex sets}
\author{Maria Carmela Ceparano\\University of Naples Federico II\\email: mariacarmela.ceparano@unina.it
\and Federico Quartieri\\University of Florence\\email: federico.quartieri@unifi.it}
\maketitle

\begin{abstract}
\noindent This paper shows that---under suitable conditions on a cone
$C$---any element in the convex hull of a decomposably $C$-antichain-convex
set $Y$ is $C$-Pareto dominated by some element of $Y$. Building on this, the
paper proves the disjointness of the convex hulls of two disjoint decomposably
$C$-antichain-convex sets whenever one of latter is $C$-upward. These findings
are used to obtain several consequences on the structure of the $C$-Pareto
optima of decomposably $C$-antichain-convex sets, on the separation of
decomposably $C$-antichain-convex sets and on the convexity of the set of
maximals of $C$-antichain-convex relations and of the set of maximizers of
$C$-antichain-quasiconcave functions. Special emphasis is placed on the
invariance of the solution set of a problem after its \textquotedblleft
convexification\textquotedblright.

\


\


\end{abstract}

\section{Introduction}

Generalized convexity plays an important role in optimization theory as well
as in its applications to mathematical economics. In this paper we deal with
two mathematical issues that are of interest to mathematical economics: the
separation of possibly non-convex sets and the convexity of the set of
maximals (resp. maximizers) of a possibly non-convex relation (resp. possibly
non-quasiconcave function).\ The first type of results can be used, for
instance, in welfare economics in order to obtain variants of the classical
\textquotedblleft Second Welfare Theorem\textquotedblright: within the
mathematical strand of literature motivated by the influential paper
\cite{Gues75}, we mention
\cite{Khan87,Khan88,Boni88,Khan99,Mord00,Boni02,Flam06,Flor06,Jofr06,Habt11,Cepa19}%
.\ The second type of results can be used, for instance, in demand theory
where the convexity of the solutions of a consumer's constrained maximization
problem is a property with useful consequences for the theory of general
equilibrium: see, e.g., \cite[Chapter 18]{Bord85}---as well as the literature
cited therein---for illustrations of the use of convex-valued (excess) demand
correspondences in the proofs of the existence of an equilibrium price. Even
if through this work we shall explain the economic motivation for the
structure of some specific optimization problems considered, the focus of the
paper is only on the mathematical aspects of such problems.

The key-notion employed in our analysis combines convexity notions with
order-theoretic ones. Any cone $C$ in a real vector space $V$ generates a
binary relation $R_{C}$ on $V$ defined by
\[
(x,y)\in R_{C}\text{ if and only if }y-x\in C\text{.}%
\]
The notion of $C$-antichain-convexity stipulates the usual definition of
convexity only for pairs of vectors that are unrelated through $R_{C}$. Such a
notion has been introduced\footnote{Under a much more particular form,
however, a variant of that definition had already appeared in \cite{Cepa17} to
prove a fixpoint theorem of Krasnose\v{l}skii type.} in \cite{Cepa19}, where
it is shown that the Minkowski sum of two $C$-antichain-convex sets is convex
if one of the two summands is $C$-upward (a set is $C$-upward whenever its
Minkowski sum with the cone $C$ is included in the set itself). This
mathematical fact---which properly generalizes the assertion that the sum of
two convex sets is convex---is used in that article to prove a separation
theorem for possibly non-convex sets.

In the present work we show two other useful properties of decomposably
$C$-antichain-convex sets (i.e., of sets that can be expressed as the
Minkowski sum of $C$-antichain-convex sets). The first---see Lemma \ref{HM}
and its generalization in Theorem \ref{MS}---is the existence, under suitable
conditions on a cone $C$ and for any element $x$ in the convex hull of a
decomposably $C$-antichain-convex set $S$, of a $C$-Pareto dominating element
$y$ in $S$ (i.e., of an element $y$ in $S$ whose difference $y-x$ with $x$ is
a vector in the cone $C$). The second---see Theorem \ref{TEOCON}---is the
disjointness of the convex hulls of two disjoint decomposably $C$%
-antichain-convex sets whenever one of latter is $C$-upward.

The first property on the existence of a $C$-Pareto dominating element has a
crucial role in proving that---in some constrained optimization problems---the
set of maximals of certain $C$-antichain-convex relations coincides with that
of their convexifications. Also, that property has an important role in the
proof that,\ under suitable condition on the cone $C$, the set of $C$-Pareto
optima of a decomposably $C$-antichain-convex set equals that of its convex
hull (this can have implications also for results on the existence of Pareto
optima like, e.g., those in \cite{Jofr15}). From the influential article
\cite{Star69}, that made use of the Shapley-Folkman theorem, many non-convex
optimization problems in mathematical economics have been tackled by
considering their \textquotedblleft convexification\textquotedblright. A
similar approach---partly motivated by the mentioned problems in mathematical
economics---has been taken in pure optimization theory: see, in particular,
\cite{Ekel99} for a discussion. The essential difference between our and those
results is that the convexity of optimal solutions is here directly obtained
for the problem under consideration, rather than as a limit or an approximated solution.

The second property has several consequences and can be suitably---though not
necessarily directly---used to extend some results of the literature about the
separation of disjoint convex sets. This paper shows that such extensions are
possible not only in the case of classical separation theorems which involve
nonempty topological interiors---or closed and compact sets---that can be
found in standard textbooks (like, e.g., \cite{Kell63}) but also in the case
of more recent results which dispense with such assumptions (like, e.g., those
in \cite{Vanc19} which use the quasi-relative interiority notion introduced by
\cite{Borw92}).

The paper is organized as follows. Sect. \ref{PRE} recalls some definitions
and shows some facts. Sect. \ref{SEC} proves the main result about the
existence of $C$-Pareto dominating elements and shows a useful consequence on
the structure of the $C$-Pareto optima of decomposably $C$-antichain-convex
sets. Sect. \ref{TRE} elaborates on the main result showing sufficient
conditions for the convex hulls of disjoint decomposably $C$-antichain-convex
sets to be disjoint. Sect. \ref{SEP} applies the previous results to obtain
various separation theorems for $C$-antichain-convex sets and Sect. \ref{MAX}
investigates on the convexity of the set of maximals (resp. maximizers) of
$C$-antichain-convex relations (resp. $C$-antichain-quasiconcave functions).

\section{\label{PRE}Preliminary definitions, notation and facts}

In this Sect. \ref{PRE} we fix the general definitions and notation used
through the entire paper and we point out some general facts. More specific
definitions and notation will be introduced and recalled at the beginning of
the sections where they will be used.

\paragraph*{Relations\ \ \ \ \ }

\noindent Let $S$ be a set. A \textbf{relation} $R$ on $S$ is a subset of
$S\times S$: when $(t,s)\in R$\ we say that $t$ is $R$\textbf{-related with
}$s$. Given a relation $R$ on $S$, the set $R(s)$ defined by
\[
R(s)=\{t\in S:(t,s)\in R\}
\]
will henceforth denote the \textbf{set of all elements of }$S$\textbf{\ that
are }$R$\textbf{-related with }$s$. A relation $R$ on $S$ is: \textbf{total}
iff for all $(s,t)\in S\times S$ we have that either $t\in R(s)$ or $s\in
R(t)$; \textbf{transitive} iff for all $(r,s,t)\in S\times S\times S$ such
that $r\in R(s)$ and $s\in R(t)$ we have $r\in R(t)$; a \textbf{total
preorder} iff $R$ is total and transitive.

\paragraph*{Maximals\newline}

\noindent Let $R$ be a relation on a set $X$ and $S$ be a subset of $X$. An
element $m\in S$ is $R$\textbf{-maximal on }$S$ iff%
\[
s\in S\text{ and }s\in R(m)\Rightarrow m\in R(s)\text{.}%
\]
The\textbf{\ set of all }$R$\textbf{-maximals on }$S$ is denoted by
$\mathcal{M}(R,S)$.

\begin{lemma}
\label{BASE}Let $R$ be\ a total relation on a set $X$. Assume that $S\subseteq
X$. Then $m\in\mathcal{M}(R,S)$ if and only if $m\in R(s)$ for all $s\in S$.
\end{lemma}

\begin{proof}
\textit{If part.} Suppose $m\in R(s)$ for all $s\in S$. The definition of a
$R$-maximal element on $S$ directly implies that $m$ is $R$-maximal on $S$. So
$m\in\mathcal{M}(R,S)$.

\noindent\textit{Only if part.} Suppose $m\in\mathcal{M}(R,S)$ and $s$ is an
arbitrary element of $S$. If $m\notin R(s)$ then the totality of $R$ implies
$s\in R(m)$ and we obtain a contradiction with the assumption that
$m\in\mathcal{M}(R,S)$. Therefore $m\in R(s)$ for all $s\in S$.
\end{proof}

\ \

Remark \ref{JOY} is an immediate consequence of Lemma \ref{BASE}.

\begin{remark}
\label{JOY}Let $X$ be a set and $S\subseteq X$. Let $u:X\rightarrow\mathbb{R}$
be a function and $R$ be the relation on $X$ defined by%
\begin{equation}
R(x)=\{y\in X:u(y)\geq u(x)\}\text{ for all }x\in X\text{.}\label{JENNY}%
\end{equation}
Then $R$ is a total preorder and $\arg\max_{s\in S}u(s)=\mathcal{M}(R,S)$.
\end{remark}

\begin{lemma}
\label{PEL}Let $R^{\bullet}$ and $R^{\circ}$ be relations on a set $X$. Assume
that $S\subseteq X$ and that $R^{\bullet}$ is total. If $R^{\bullet}\subseteq
R^{\circ}$ then $\mathcal{M}(R^{\bullet},S)\subseteq\mathcal{M}(R^{\circ},S)$.
\end{lemma}

\begin{proof}
Assume that $R^{\bullet}\subseteq R^{\circ}$ and $m\in\mathcal{M}(R^{\bullet
},S)$. By way of contradiction, suppose $m\notin\mathcal{M}(R^{\circ},S)$.
Then there exists $s\in S$ such that $s\in R^{\circ}(m)$ and $m\notin
R^{\circ}(s)$. Lemma \ref{BASE} ensures that $m\in R^{\bullet}(s)$. As
$m\notin R^{\circ}(s)$, the assumption that $R^{\bullet}\subseteq R^{\circ}$
implies $m\notin R^{\bullet}(s)$: a contradiction with $m\in R^{\bullet}(s)$.
\end{proof}

\begin{lemma}
\label{AIUT}Suppose $R$ is\ a total preorder on a set $X$. Assume that
$S\subseteq X$, that $x\in S$ and that $y\in\mathcal{M}(R,S)$.

\begin{enumerate}
\item If $x\in R(y)$ then $x\in\mathcal{M}(R,S)$.

\item If $s\in S$ and $s\notin R(x)$ then $y\in R(s)$ and $s\notin R(y)$.
\end{enumerate}
\end{lemma}

\begin{proof}
1. Suppose $x\in R(y)$. As $y\in\mathcal{M}(R,S)$, Lemma \ref{BASE} ensures
that $y\in R(s)$ for all $s\in S$. The transitivity of $R$ then implies that
$x\in R(s)$ for all $s\in S$. Noted this, Lemma \ref{BASE} ensures that
$x\in\mathcal{M}(R,S)$.

2. Suppose $s\in S$ and $s\notin R(x)$. Then Lemma \ref{BASE} ensures that
$y\in R(s)$ and $s\notin\mathcal{M}(R,S)$. As $s\notin\mathcal{M}(R,S)$, part
1 of Lemma \ref{AIUT} implies $s\notin R(y)$.
\end{proof}

\paragraph*{Spaces and operators\newline}

\noindent A real vector space is henceforth abbreviated by RVS. Analogously, a
topological real vector space is abbreviated by TRVS and a locally convex
topological real vector space by LCS. When $V$ is a TRVS, its topological dual
is henceforth denoted by $V^{\ast}$. The Minkowski sum of two subsets $A$ and
$B$ of a RVS is denoted by $A+B$ and we simply denote by $A-B$ the Minkowski
sum of $A$ and $-B$; when $A=\{a\}$ we simply write $a+B$ instead of the more
cumbersome $\{a\}+B$. Let $V$ be a RVS, a subset $C$ of $V$ is a \textbf{cone
in }$V$ iff $\lambda C\subseteq C$ for all $\lambda\in\mathbb{R}_{++}$ (and
hence a cone need not be nonempty and need not contain the zero vector). A
cone $C$ in a RVS is a \textbf{pointed cone} iff $C\cap-C\subseteq\{0\}$.

\begin{remark}
\label{TENTO}A cone $C$ in a RVS is convex if and only if $C+C\subseteq C$.
\end{remark}

\begin{lemma}
\label{ERRA}Let $V$ be a RVS and $C$ be a convex cone in $V$.

\begin{enumerate}
\item Then $C\cup\{0\}$ is a convex cone.

\item Suppose $C$ is pointed. Then $C\backslash\{0\}$ is a convex cone.
\end{enumerate}
\end{lemma}

\begin{proof}
1. Put $C_{0}=C\cup\{0\}$. Part 1 of Lemma 1 in \cite{Cepa19} guarantees that
$C_{0}$ is a cone in $V$. We are done if we show that $C_{0}+C_{0}\subseteq
C_{0}$. Noting that
\[
C_{0}+C_{0}=(C\cup\{0\})+(C\cup\{0\})=(C+C)\cup C\cup C\cup\{0\}
\]
because of basic properties of the Minkowski sum and that $C+C\subseteq C$ by
the convexity of the cone $C$, we conclude that $C_{0}+C_{0}\subseteq C_{0}$.

2. The fact that $C\backslash\{0\}$ is a cone follows immediately from the
definition of a cone. To prove that $C\backslash\{0\}$ is convex just note
that the convex combination of any two points in $C\backslash\{0\}$ belongs to
$C$ by the convexity of $C$ and that it cannot be equal to the vector zero as
$C$ is pointed.
\end{proof}

\

Let $V$ be a RVS and $X$ be a subset of $V$. We denote by $\operatorname*{co}%
(X)$ (resp. $\operatorname*{cone}(X)$, $\operatorname*{aff}(X)$) the convex
hull (resp. the conic hull, the affine hull) of $X$ recalling that
$\operatorname*{cone}(X)=\{\lambda x:(\lambda,x)\in\mathbb{R}_{+}\times X\}$
and hence that $\operatorname*{cone}(X)$ contains the zero vector whenever
$X\neq\emptyset$. When $V$ is endowed with some topology, we denote by
$\operatorname*{bd}(X)$ (resp. $\operatorname*{int}(X)$, $\operatorname*{cl}%
(X)$) the boundary (resp. the topological interior, the topological closure)
of $X$.\ \ \

\paragraph*{Cone-based relational and convexity notions\ }

\noindent Let $V$ be a RVS, $C$ be a cone in $V$ and $S$ be a subset of $V$.
The following four definitions have been introduced in \cite{Cepa19}---and we
refer to the mentioned article for a discussion---while the last definition is
a new cone-based extension of the notion of an antichain. We say that:

\begin{itemize}
\item $S$ is $C$\textbf{-antichain-convex} iff
\[
(x,y,\lambda)\in S\times S\times\lbrack0,1]\text{ and }y-x\notin
C\cup-C\Rightarrow\lambda x+(1-\lambda)y\in S\text{;}%
\]

\item $S$ is \textbf{decomposably }$C$\textbf{-antichain-convex} iff there
exists a finite collection $\{S_{1},\ldots,S_{n}\}$ of $C$-antichain-convex
subsets of $V$ such that%
\[
S=S_{1}+\ldots+S_{n}\text{;}%
\]

\item $S$ is $C$\textbf{-upward} iff
\[
(x,y)\in S\times S\text{ and }y-x\in C\Rightarrow y\in S\text{;}%
\]

\item $S$ is $C$\textbf{-downward} iff
\[
(x,y)\in S\times S\text{ and }x-y\in C\Rightarrow y\in S\text{;}%
\]

\item $S$ is a $C$\textbf{-antichain} iff%
\[
(x,y)\in S\times S\text{ and }x\neq y\text{ }\Rightarrow\text{ }y-x\notin
C\cup-C\text{.}%
\]

\end{itemize}

\noindent Recall that when $C\subseteq\{0\}$ any convex set is both
$C$-antichain-convex and $C$-upward (as well as $C$-downward).

\begin{remark}
\label{K2} Let $V$ be a RVS, $C$ be a cone in $V$ and $S$ be a subset of $V$.
Noting that $-C$ is a cone, from the previous definitions it follows that $S$
is (decomposably) $C$-antichain-convex if and only if $S$ is (decomposably)
$-C$-antichain-convex and that $S$ is $C$-downward if and only if $S$ is $-C$-upward.
\end{remark}

The following facts are two general results of some importance for the sequel
of this work. Theorem \ref{T:CH} is essentially new while Theorem
\ref{STANGHI} uses---but does not directly follow from---a previous result
shown in \cite{Cepa19}.

\begin{theorem}
\label{T:CH} Let $V$ be a RVS, $C$ be a cone in $V$ and $S$ be a subset of $V
$.

\begin{enumerate}
\item $S$ is $\operatorname*{co}(C)$-upward if and only if $S$ is $C$-upward.

\item $S$ is $\operatorname*{co}(C)$-downward if and only if $S$ is $C$-downward.
\end{enumerate}
\end{theorem}

\begin{proof}
1. As $C\subseteq\operatorname*{co}(C)$, if $S$ is $\operatorname*{co}%
(C)$-upward then part 2 of Proposition 6 in \cite{Cepa19} guarantees that $S$
is $C$-upward. So henceforth suppose $S$ is $C$-upward. We are done if we show
that $S$ is $\operatorname*{co}(C)$-upward. When either $S$ or $C$ is empty,
the assertion is trivially true. So suppose $S$ and $C$ are nonempty and pick
$s\in S$ and $c\in\operatorname*{co}(C)$. Thus there exist $n$ elements
$c_{1},\ldots,c_{n}$ in $C$ and $\alpha$ in $\mathbb{R}_{++}^{n}$ such that
$\alpha_{1}+\ldots+\alpha_{m}=1$ and $c=\alpha_{1}c_{1}+\ldots+\alpha_{n}%
c_{n}$. As $C$ is a cone, we have that $\alpha_{i}c_{i}\in C{\text{ for all }%
}i=1,\ldots,n$. Put $x_{0}=s$ and $x_{i}=x_{i-1}+\alpha_{i}c_{i}$ for all
$i=1,\ldots,n$. Reasoning by induction, note that $x_{i}\in S$ for all
$i=1,\ldots,n$ as it is the sum of the element $x_{i-1}$ of the $C$-upward set
$S$ and of the element $\alpha_{i}c_{i}$ of the cone $C$. Being $x_{n}=s+c$,
we conclude that $S$ is $\operatorname*{co}(C)$-upward.

2. A consequence of part 1 of Theorem \ref{T:CH}, Remark \ref{K2} and the fact
that $\operatorname*{co}(-C)=-\operatorname*{co}(C)$.
\end{proof}

\begin{theorem}
\label{STANGHI}Let $V$ be a RVS and $C$ be a cone in $V$. Suppose $X$ is a
decomposably $C$-antichain-convex subset of $V$.

\begin{enumerate}
\item If $X$ is $C$-upward then $X$ is convex.

\item If $X$ is $C$-downward then $X$ is convex.
\end{enumerate}
\end{theorem}

\begin{proof}
1.\ Suppose $X$ is $C$-upward.\ The decomposable $C$-antichain-convexity of
$X$ implies the existence of $C$-antichain-convex subsets $X_{1},\ldots,X_{n}
$ of $V$ such that $X_{1}+\ldots+X_{n}=X$. Put%
\[
C_{0}=C\cup\{0\}\text{ and }K=\operatorname*{co}(C_{0}).
\]
Part 6 of Proposition 3 in \cite{Cepa19} and the inclusion $C\subseteq K$
guarantee that $X_{1},\ldots,X_{n}$ are $K$-antichain-convex. Part 1 of Lemma
4 in \cite{Cepa19} and part 1 of Theorem \ref{T:CH} guarantee that $X$ is
$K$-upward. As $X$ is $K$-upward and $K$ contains the zero vector, part 3 of
Lemma 5 in \cite{Cepa19} implies $X+K=X$ and hence%
\[
X=X_{1}+\ldots+X_{n}+K\text{.}%
\]
So $X$ can be expressed as the sum of $n+1$ sets which are $K$%
-antichain-convex subsets of $V$. Note that $K$ (i.e., the last of the $n+1$
addends) is also $K$-upward: this is a consequence of the fact that $K$ is a
convex cone (see Remark \ref{TENTO}). So $X$ is convex by part 1 of Corollary
1 in \cite{Cepa19}.

2. A consequence of part 1 of Theorem \ref{STANGHI} and Remark \ref{K2}.
\end{proof}

\section{\label{SEC}Existence of $C$-Pareto dominating elements}

\subsection{Preliminary definitions}

Let $V$ be a RVS, $C$ be a cone in $V$ and\ $Y$\ be a subset of $V$. An
element $v$ of $V$ is $C$\textbf{-Pareto dominated} by an element $z$ of $V$
iff $z\in v+C$. When $v\in V$ is $C$-Pareto dominated by $z\in V$, we say that
$z$ is $C$\textbf{-Pareto dominating }$v$. An element $y$ of $Y$ is a $C
$\textbf{-Pareto optimum of }$Y$ iff the sets $Y\backslash\{y\}$ and $y+C$ are
disjoint. The \textbf{set of }$C$\textbf{-Pareto optima of} $Y$ is denoted by
$\mathcal{O}(C,Y)$.

\begin{remark}
The notion of a $C$-Pareto optimum is not new. For instance, when $C$ is a
cone containing the zero vector, the definition of a $C$-Pareto optimum boils
down to that of \emph{Pareto optimality with respect to }$C$ in
\cite[Definition 1]{Jofr15} and, when $C$ is a closed convex cone containing
the zero vector, our definition is the exact \textquotedblleft
dual\textquotedblright\ of that of a \emph{Pareto minimal} point given in
\cite[Definition 9.1]{Mord18}. On the relation with other nonequivalent
definitions---and on the implications of the pointedness of $C$---see the
discussion at p. 452 in \cite{Mord18}.
\end{remark}

\subsection{Main result}

Lemma \ref{HM} guarantees the existence of $C$-Pareto dominating elements.
Lemma \ref{HM}, as generalized by Theorem \ref{MS}, is the main finding of the
paper and plays an important role in many subsequent results.

\begin{lemma}
\label{HM}Let $V$ be a RVS and $C$ be a convex cone in $V$ containing the zero
vector. Suppose $Y$ is a $C$-antichain-convex subset of $V$.

\begin{enumerate}
\item Each $y$ in $\operatorname*{co}(Y)$ is $C$-Pareto dominated by at least
one $z$ in $Y$.

\item Each $y$ in $\operatorname*{co}(Y)$ is $C$-Pareto dominating at least
one $x$ in $Y$.
\end{enumerate}
\end{lemma}

\begin{proof}
1. Part 1 of Lemma \ref{HM} is true if we show the validity of the following
(equivalent) assertion.

\ \ \ \ \ \ %

\begin{tabular}
[c]{p{1.5cm}}%
\textbf{Assertion.}%
\end{tabular}%
\begin{tabular}
[c]{p{8.8cm}}%
\textit{If }$y$\textit{\ can be expressed as the convex combination of }%
$n\in\mathbb{N}$\textit{\ elements of }$Y$\textit{\ then there exists }%
$z$\textit{\ in }$Y$\textit{\ such that }$z\in y+C$\textit{.}%
\end{tabular}

\ \

\noindent The Assertion is trivially true if $n=1$ in that $0\in C$. The rest
of the proof is by induction. Let $m\in\mathbb{N}\backslash\{1\}$ and assume
as an induction hypothesis that the Assertion is true if $n<m-1$. We show that
the Assertion is true when $n=m$. Assume that $n=m$ and that $y$ can be
expressed as the convex combination of $m$ elements of $Y$. Then there exist
$y_{1},\ldots,y_{m}$ in $Y$ and $\alpha$ in $\mathbb{R}_{+}^{m} $ such that
$\alpha_{1}+\ldots+\alpha_{m}=1$ and%
\[
y=\alpha_{1}y_{1}+\ldots+\alpha_{m}y_{m}\text{.}%
\]
Because of the induction hypothesis, we can assume without loss of generality
that $\alpha_{i}>0$ for all $i=1,\ldots,m$. Put%
\[
y_{0}=\frac{\alpha_{2}y_{2}+\ldots+\alpha_{m}y_{m}}{1-\alpha_{1}}\text{ and
}\alpha_{0}=1-\alpha_{1}\text{.}%
\]
Note that%
\[
y=\alpha_{0}y_{0}+(1-\alpha_{0})y_{1}\text{.}%
\]
and that $y_{0}$ is the convex combination of the $m-1$ elements $y_{2}%
$,\ldots,$y_{m}$ with coefficients $\alpha_{2}/(1-\alpha_{1})$,\ldots
,$\alpha_{m}/(1-\alpha_{1})$. Then the induction hypothesis guarantees the
existence of $c_{0}\in C$ such that%
\[
y_{0}+c_{0}\in Y\text{.}%
\]

\noindent$\bullet$\quad If $y_{0}+c_{0}\notin y_{1}+(C\cup-C)$ then
$\alpha_{0}(y_{0}+c_{0})+(1-\alpha_{0})y_{1}\in Y$ by the $C$%
-antichain-convexity of $Y$. So%
\[
y+\alpha_{0}c_{0}=\alpha_{0}y_{0}+(1-\alpha_{0})y_{1}+\alpha_{0}c_{0}\in
Y\text{.}%
\]
Noting that $\alpha_{0}c_{0}$ is an element of the cone $C$, we conclude that
the Assertion is true: just identify $z$ with $y+\alpha_{0}c_{0}$.

\noindent$\bullet$\quad If $y_{0}+c_{0}\in y_{1}+C$ then $y_{0}+c_{0}%
=y_{1}+c_{1}$ for some $c_{1}\in C$. So%
\[
(1-\alpha_{0})(y_{0}+c_{0})=(1-\alpha_{0})(y_{1}+c_{1})\text{.}%
\]
Adding the vector $\alpha_{0}(y_{0}+c_{0})$ to both sides of the previous
equality we get%
\[
y_{0}+c_{0}=\alpha_{0}y_{0}+(1-\alpha_{0})y_{1}+\alpha_{0}c_{0}+(1-\alpha
_{0})c_{1}=y+\alpha_{0}c_{0}+(1-\alpha_{0})c_{1}\text{.}%
\]
The convexity of the cone $C$ implies $\alpha_{0}c_{0}+(1-\alpha_{0})c_{1}\in
C$ and hence%
\[
y_{0}+c_{0}\in y+C\text{.}%
\]
We conclude that the Assertion is true: just identify $z$ with $y_{0}+c_{0}$.

\noindent$\bullet$\quad If $y_{0}+c_{0}\in y_{1}-C$ then $y_{0}+c_{0}%
=y_{1}-c_{1}$ for some $c_{1}\in C$. So%
\[
y_{1}-y_{0}=c_{0}+c_{1}\text{.}%
\]
The convexity of $C$ then implies $c_{0}+c_{1}\in C$ and hence $y_{1}-y_{0}\in
C$. Therefore $\alpha_{0}(y_{1}-y_{0})\in C$. Note that $y_{1}-y=y_{1}%
-(\alpha_{0}y_{0}+(1-\alpha_{0})y_{1})=\alpha_{0}(y_{1}-y_{0})$ and hence that%
\[
y_{1}-y\in C\text{.}%
\]
So $y_{1}=y+c$ for some $c\in C$. We conclude that the Assertion is true: just
identify $z$ with $y_{1}$.

2. A consequence of part 1 of Lemma \ref{HM} and Remark \ref{K2}.
\end{proof}

\begin{theorem}
\label{MS}Let $V$ be a RVS and $C$ be a cone in $V$. Suppose $Y$ is a
decomposably $C$-antichain-convex subset of $V$ and put $K=\operatorname*{co}%
(C\cup\{0\})$.

\begin{enumerate}
\item Each $y$ in $\operatorname*{co}(Y)$ is $K$-Pareto dominated by at least
one $z$ in $Y$.

\item Each $y$ in $\operatorname*{co}(Y)$ is $K$-Pareto dominating at least
one $x$ in $Y$.
\end{enumerate}
\end{theorem}

\begin{proof}
1. The assumption that $Y$ is decomposably $C$-antichain-convex implies the
existence of $C$-antichain-convex subsets $Y_{1},\ldots,Y_{n}$ of $V$ such
that $Y=Y_{1}+\ldots+Y_{n}$. The set $K$ is a convex cone in $V$ containing
the zero vector. As $C\subseteq K$, part 6 of Proposition 3 in \cite{Cepa19}
ensures that $Y_{1},\ldots,Y_{n}$ are $K$-antichain-convex. Suppose
$y\in\operatorname*{co}(Y)$. As the convex hull of the Minkowski sum of $n$
sets equals the Minkowski sum of their convex hulls, we have that
\[
y\in\operatorname*{co}(Y_{1})+\ldots+\operatorname*{co}(Y_{n})\text{.}%
\]
So there exists a tuple $(y_{1},\ldots,y_{n})$ in $\operatorname*{co}%
(Y_{1})\times\ldots\times\operatorname*{co}(Y_{n})$ such that $y=y_{1}%
+\ldots+y_{n}$. For each $i\in\{1,\ldots,n\}$, part 1 of Lemma \ref{HM}
implies the existence of $z_{i}\in Y_{i}$ such that $y_{i}$ is $K$-Pareto
dominated by $z_{i}$. Thus $z_{i}-y_{i}\in K$ for all $i\in\{1,\ldots,n\}$ and
so the convexity of the cone $K$ and Remark \ref{TENTO} together imply
\begin{equation}%
{\textstyle\sum\nolimits_{i=1}^{n}}
(z_{i}-y_{i})\in K\text{.}\label{KDE}%
\end{equation}
Put $z=z_{1}+\ldots+z_{n}$ and note that $z\in Y$ and that (\ref{KDE}) is
equivalent to $z-y\in K$. Therefore $z\in y+K$ and hence $y$ is $K$-Pareto
dominated by $z\in Y$.

2. A consequence of part 1 of Theorem \ref{MS} and Remark \ref{K2}.
\end{proof}

\subsection{\label{PPPO}On $C$-Pareto optima}

Theorem \ref{EQUO} shows a result of independent interest elucidating on the
maximality of Pareto optima. Theorem \ref{REZZINA} is the main result of this
Sect. \ref{PPPO}. It guarantees the equivalence of the set of $C$-Pareto
optima of a decomposably $C$-antichain-convex set $Y$ and that of its convex
hull when $C$ is a pointed convex cone. The last result has consequences of
interest for economic theory.\footnote{To provide a tangible example of an
implication of economic interest, note that part 2 of Theorem \ref{REZZINA}
allows us to generalize Proposition 5.F.2 in \cite{Masc95} by replacing in its
statement \textquotedblleft convex\textquotedblright\ with \textquotedblleft
decomposably $\mathbb{R}_{+}^{n}$-antichain-convex\textquotedblright\ since
any production vector that is profit-maximizing on $\operatorname*{co}(Y)$
must be profit-maximizing on $Y$.}

\begin{theorem}
\label{EQUO}Let $V$ be a RVS, $C$ be a cone in $V$ and $Y$ be a subset of $V$.
Let $D$ be the relation on $Y$ defined by%
\[
D(y)=\{x\in Y:y\text{ is }C\text{-Pareto dominated by }x\}\text{. }%
\]

\begin{enumerate}
\item Then $\mathcal{O}(C,Y)\subseteq\mathcal{M}(D,Y)$.

\item Assume that $C$ is pointed. Then $\mathcal{O}(C,Y)=\mathcal{M}(D,Y)$.
\end{enumerate}
\end{theorem}

\begin{proof}
1. Suppose $y^{\ast}\in\mathcal{O}(C,Y)$. If $y^{\ast}\notin\mathcal{M}(D,Y)$
then there exists $y\in Y$ such that $y\in D(y^{\ast})$ and $y^{\ast}\notin
D(y)$: therefore $y^{\ast}\neq y\in Y$ and $y\in y^{\ast}+C$ getting a
contradiction with the assumption that $y^{\ast}\in\mathcal{O}(C,Y)$. So
$y^{\ast}\in\mathcal{M}(D,Y)$.

2. By virtue of part 1 of Theorem \ref{EQUO}, we are done if we show that
$\mathcal{M}(D,Y)\subseteq\mathcal{O}(C,Y)$. Its proof is as follows. Assume
that $m\in\mathcal{M}(D,Y)$, that $x\in Y$ and that $x\in m+C$: we conclude
showing that $x=m$. The definition of $D$ and the membership $x\in m+C$
implies $x\in D(m)$: the $D$-maximality of $m$ in turn entails that $m\in
D(x)$ and so that $m\in x+C$. The memberships $x\in m+C$ and $m\in x+C$ imply
$x=m$ by the pointedness of $C$.
\end{proof}

\begin{lemma}
\label{EKKA}Let $V$ be a RVS, $C$ be a cone in $V$ and $Y$ be a subset of $V$.
Then $\mathcal{O}(C,Y)=\mathcal{O}(C\cup\{0\},Y)$.
\end{lemma}

\begin{proof}
Put $C_{0}=C\cup\{0\}$. Part 1 of Lemma 1 in \cite{Cepa19} guarantees that
$C_{0}$ is a cone in $V$. Noted this, Lemma \ref{EKKA} is a consequence of the
observation that $Y\backslash\{y\}$ and $y+C$ are disjoint if and only if so
are $Y\backslash\{y\}$ and $y+C_{0}$.
\end{proof}

\begin{theorem}
\label{REZZINA}Let $V$ be a RVS, $C$ be a convex cone in $V$ and $Y$ be a
decomposably $C$-antichain-convex subset of $V$.

\begin{enumerate}
\item Then $\mathcal{O}(C,\operatorname*{co}(Y))\subseteq\mathcal{O}(C,Y)$.

\item Assume that $C$ is pointed. Then $\mathcal{O}(C,Y)=\mathcal{O}%
(C,\operatorname*{co}(Y))$.
\end{enumerate}
\end{theorem}

\begin{proof}
Put $C_{0}=C\cup\{0\}$. The set $C_{0}$ is a convex cone by part 1 of Lemma
\ref{ERRA}. So $C_{0}=\operatorname*{co}(C\cup\{0\})$.

1. By virtue of Lemma \ref{EKKA}, the inclusion $\mathcal{O}%
(C,\operatorname*{co}(Y))\subseteq\mathcal{O}(C,Y)$ is equivalent to the
inclusion $\mathcal{O}(C_{0},\operatorname*{co}(Y))\subseteq\mathcal{O}%
(C_{0},Y)$. To conclude the proof we show the validity of the last inclusion,
as follows. Suppose $y\in\mathcal{O}(C_{0},\operatorname*{co}(Y))$. We are
done if we show that $y\in Y$. As $y\in\operatorname*{co}(Y)$, part 1 of
Theorem \ref{MS} implies the existence of $z\in Y$ such that $z\in y+C_{0}$.
As $y\in\mathcal{O}(C_{0},\operatorname*{co}(Y))$, we must have that $y=z$. So
$y\in Y$.

2. Part 1 of Theorem \ref{REZZINA} ensures that $\mathcal{O}%
(C,\operatorname*{co}(Y))\subseteq\mathcal{O}(C,Y)$. So we are done if we
prove that $\mathcal{O}(C,Y)\subseteq\mathcal{O}(C,\operatorname*{co}(Y))$. By
virtue of Lemma \ref{EKKA}, we are done if we prove that%
\[
\mathcal{O}(C_{0},Y)\subseteq\mathcal{O}(C_{0},\operatorname*{co}(Y))\text{.}%
\]
The proof of the last inclusion is as follows. Suppose $y\in\mathcal{O}%
(C_{0},Y)$ and, by way of contradiction, suppose $y\notin\mathcal{O}%
(C_{0},\operatorname*{co}(Y))$: then there exists $c\neq0$ in $C_{0}$ such
that $y+c\in\operatorname*{co}(Y)$. Part 1 of Theorem \ref{MS} implies the
existence of $z\in Y $ such that $z\in y+c+C_{0}$. Note that the convex cone
$C_{0}$ is pointed and hence that $c+C_{0}\subseteq C_{0}\backslash
\{0\}=C\backslash\{0\}$ as $c\neq0$. The membership $z\in y+c+C_{0}$ and the
inclusion $c+C_{0}\subseteq C\backslash\{0\}$ imply $z\in y+C_{0}$ and $z\neq
y$: as $z\in Y$, we have obtained a contradiction with the assumption that
$y\in\mathcal{O}(C_{0},Y)$.
\end{proof}

\section{\label{TRE}On disjoint convex hulls}

Lemma \ref{MT1} guarantees that the convex hulls of disjoint $C$%
-antichain-convex sets are disjoint whenever one of the $C$-antichain-convex
sets is $C$-upward (or, dually, $C$-downward). Theorem \ref{TEOCON}
generalizes to decomposably $C$-antichain-convex sets.

\begin{lemma}
\label{MT1}Let $V$ be a RVS and $C$ be a convex cone in $V$ containing the
zero vector. Suppose $X$ and $Y$ are disjoint $C$-antichain-convex subsets of
$V$.

\begin{enumerate}
\item Suppose $X$ is $C$-upward. Then $\operatorname*{co}(X)$ and
$\operatorname*{co}(Y)$ are disjoint.

\item Suppose $X$ is $C$-downward. Then $\operatorname*{co}(X)$ and
$\operatorname*{co}(Y)$ are disjoint.
\end{enumerate}
\end{lemma}

\begin{proof}
1. Proposition 2 in \cite{Cepa19} and part 1 of Proposition 5 in \cite{Cepa19}
guarantee that $X$ is convex. So $X=\operatorname*{co}(X)$ and we are done if
we prove that $X\cap\operatorname*{co}(Y)=\emptyset$. By way of contradiction,
suppose $X\cap\operatorname*{co}(Y)\neq\emptyset$ and pick $y\in
X\cap\operatorname*{co}(Y)$. As $y\in X$ and $X$ is $C$-upward, part 1 of
Lemma 5 in \cite{Cepa19} implies that
\begin{equation}
y+C\subseteq X\text{.}\label{Gatto0}%
\end{equation}
As $y\in\operatorname*{co}(Y)$, part 1 of Lemma \ref{HM} guarantees the
existence of
\begin{equation}
x\in Y\label{Gatto1}%
\end{equation}
such that $x\in y+C$. The previous membership and (\ref{Gatto0}) imply $x\in
X$: a contradiction with (\ref{Gatto1}) and the assumption that $X$ and $Y$
are disjoint.

2. A consequence of part 1 of Lemma \ref{MT1} and Remark \ref{K2}.
\end{proof}

\begin{theorem}
\label{TEOCON}Let $V$ be a RVS and $C$ be a cone in $V$. Assume that $X$ and
$Y$ are decomposably $C$-antichain-convex subsets of $V$.

\begin{enumerate}
\item Suppose $X$ is $C$-upward. Then $\operatorname*{co}(X)$ and
$\operatorname*{co}(Y)$ are disjoint.

\item Suppose $X$ is $C$-downward. Then $\operatorname*{co}(X)$ and
$\operatorname*{co}(Y)$ are disjoint.
\end{enumerate}
\end{theorem}

\begin{proof}
1. Part 1 of Theorem \ref{STANGHI} ensures that $X$ is convex. The assumption
that $Y$ is decomposably $C$-antichain-convex implies the existence of
$C$-antichain-convex subsets $Y_{1}$,$\ldots$,$Y_{n}$ of $V$ such that
$Y=Y_{1}+\ldots+Y_{n}$. Put
\[
Z=X+(-Y_{1})+\ldots+(-Y_{n-1})
\]
and note that $Z\cap Y_{n}=\emptyset$.\footnote{\label{CINIC copy(2)}Letting
$D$, $E$ and $F$ be subsets of a RVS, to prove the previous equality note that
$D\cap(E+F)=\emptyset$ if and only if $(D-E)\cap F=\emptyset$.} The sets
$-Y_{1}$,$\ldots$,$-Y_{n-1}$ are $C$-antichain-convex by part 6 of Lemma 3 in
\cite{Cepa19}. So $Z$ is convex and $C$-upward by part 1 of Corollary 1 in
\cite{Cepa19} and hence $Z=\operatorname*{co}(Z)$. Put $C_{0}=C\cup\{0\}$ and
$K=\operatorname*{co}(C_{0})$. Part 6 of Proposition 3 in \cite{Cepa19} and
the inclusion $C\subseteq K$ imply the $K$-antichain-convexity of $Y_{n}$.
Part 1 of Lemma 4 in \cite{Cepa19} and part 1 of Theorem \ref{T:CH} guarantee
that $Z$ is $K$-upward. So part 1 of Lemma \ref{MT1} ensures that
$\operatorname*{co}(Z)\cap\operatorname*{co}(Y_{n})=\emptyset$ and the
equality $Z=\operatorname*{co}(Z)$ in turn implies%
\begin{equation}
Z\cap\operatorname*{co}(Y_{n})=\emptyset\text{.}\label{UKAA}%
\end{equation}
The definition of $Z$ and the fact that the convex hull of the Minkowski sum
of $n$ sets equals the Minkowski sum of their convex hulls entail that
\begin{equation}
Z=\operatorname*{co}(Z)=\operatorname*{co}(X)-(\operatorname*{co}%
(Y_{1})+\ldots+\operatorname*{co}(Y_{n-1}))\text{.}\label{UMII}%
\end{equation}
The equalities in (\ref{UKAA}) and (\ref{UMII}) yield%
\[
(\operatorname*{co}(X)+(-\operatorname*{co}(Y_{1}))+\ldots
+(-\operatorname*{co}(Y_{n-1})))\cap\operatorname*{co}(Y_{n})=\emptyset
\]
and so $\operatorname*{co}(X)\cap(\operatorname*{co}(Y_{1})+\ldots
+\operatorname*{co}(Y_{n}))=\emptyset$. As the Minkowski sum of the convex
hulls of $n$ sets equals the convex hull of their Minkowski sum, the previous
equality implies $\operatorname*{co}(X)\cap\operatorname*{co}(Y)=\emptyset$. \

2. A consequence of part 1 of Theorem \ref{TEOCON} and Remark \ref{K2}.
\end{proof}

\ \ \ \ \

Corollary \ref{Coro3} is a consequence of Theorem \ref{TEOCON}. The motivation
for explicitly considering $X$ and $Y$ as the Minkowski sums of sets is due to
the possible applications of this type of results. We shall return on this
point in Sect. \ref{SEP}.

\begin{corollary}
\label{Coro3}Let $V$ be a RVS and $C$ be a cone in $V$. Assume that $X_{1}%
$,\ldots, $X_{m}$, $Y_{1}$,\ldots, $Y_{n}$ are $C$-antichain-convex subsets of
$V$. Put%
\[
X=X_{1}+\ldots+X_{m}\text{ \ and }Y=Y_{1}+\ldots+Y_{n}%
\]
and assume that $X$ and $Y$ are disjoint.

\begin{enumerate}
\item Suppose $X_{1}$ is $C$-upward. Then $\operatorname*{co}(X)$ and
$\operatorname*{co}(Y)$ are disjoint.

\item Suppose $X_{1}$ is $C$-downward. Then $\operatorname*{co}(X)$ and
$\operatorname*{co}(Y)$ are disjoint.
\end{enumerate}
\end{corollary}

\begin{proof}
A consequence of Corollary 1 in \cite{Cepa19} and Theorem \ref{TEOCON}.
\end{proof}

\section{\label{SEP}Separation}

Some of the previous results are now applied to obtain separation theorems
that dispense with the assumption of convexity (at least for one of the two
separated sets). All our applications hinge on---and extend---known theorems
of the literature about the separation of convex sets. After recalling some
definitions in Sect. \ref{PRESEP}, we subsequently present new results on the
separation of two decomposably $C$-antichain-convex sets. Results about the
separation of not necessarily convex sets that are Minkowski sums of other
sets is of interest in economics as they allow to extend the classical Second
Welfare Theorem(s) for convex economies to economies with non-convexities (we
refer to \cite{Cepa19} for a longer discussion and for a concrete application
of a version of a separation theorem similar---albeit nonequivalent---to that
presented in Sect. \ref{NTI}). The result in Sect. \ref{NTI} posits the
nonemptiness of the topological interior of one of the separated sets: in
Sect. \ref{CAC} and \ref{QQRRII} the applications dispense with such an
assumption. In Sect. \ref{QQRRII} we build on a separation theorem due to
\cite{Vanc19} which employs the quasi-relative interiority notion. Other
results and discussions on the separation of convex sets that use the
quasi-relative interiority notion can be found, for instance, in
\cite{Camm05,Dani07,BotC08,Zali15}.

\subsection{\label{PRESEP}Preliminary definitions}

Let $V$ be a RVS endowed with some topology, let $X$, $Y$ and $Z$ be subsets
of $V$ and let $f\in V^{\ast}$. We say that: $f$ \textbf{separates} $X$ and
$Y$ iff $\sup f[X]\leq\inf f[Y]$; $f$ \textbf{properly separates} $X$ and $Y $
iff $f$ separates $X$ and $Y$ and $\inf f[X]<\sup f[Y]$; $f$ \textbf{strictly
separates} $X$ and $Y$ iff $\sup f[X]<\inf f[Y]$.\footnote{Consequently: $f$
separates $X$ and $Y$ iff $f(x)\leq g(y)$ for all $(x,y)\in X\times Y$; $f$
properly separates $X$ and $Y$ iff $f(x)\leq g(y)$ for all $(x,y)\in X\times
Y$ and there exists $(x_{0},y_{0})\in X\times Y$ such that $f(x_{0})<g(y_{0}%
)$; $f$ strictly separates $X$ and $Y$ iff $f(x)<g(y)$ for all $(x,y)\in
X\times Y$.} Moreover, we say that: $X$ and $Y$ are \textbf{separated} iff
there exists $f\in V^{\ast}\backslash\{0\}$ that separates $X$ and $Y$; $X$
and $Y$ are \textbf{properly separated} iff there exists $f\in V^{\ast}$ that
properly separates $X$ and $Y$; $X$ and $Y $ are \textbf{strictly separated}
iff there exists $f\in V^{\ast}$ that strictly separates $X$ and $Y$. Assuming
the convexity of $Z$---and following the notation in \cite{Mord18} and
\cite{Vanc19}---we say that:

\begin{itemize}
\item the \textbf{relative interior} of $Z$ is the set
\[
\operatorname*{ri}(Z)=\{z\in Z:\exists\text{ a neighborhood }\mathrm{U}\text{
of }z\text{ such that }\mathrm{U}\cap\operatorname*{cl}(\operatorname*{aff}%
(Z))\subseteq Z\}\text{;}%
\]

\item the \textbf{intrinsic relative interior} of $Z$ is the set
\[
\operatorname*{iri}(Z)=\{z\in Z:\operatorname*{cone}(Z-z)\text{ is a subspace
of }V\}\text{;}%
\]

\item the \textbf{quasi-relative interior} of $Z$ is the set
\[
\operatorname*{qri}(Z)=\{z\in Z:\operatorname*{cl}(\operatorname*{cone}%
(Z-z))\text{ is a subspace of }V\}\text{;}%
\]

\item the set $Z$ is \textbf{quasi-regular} iff $\operatorname*{iri}%
(Z)=\operatorname*{qri}(Z)$.
\end{itemize}

\noindent\noindent Finally, we recall the following known facts:

\begin{itemize}
\item $\operatorname*{int}(Z)\subseteq\operatorname*{ri}(Z)\subseteq
\operatorname*{iri}(Z)\subseteq\operatorname*{qri}(Z)$;

\item $\operatorname*{int}(Z)=\operatorname*{ri}(Z)=\operatorname*{iri}%
(Z)=\operatorname*{qri}(Z)$ whenever $\operatorname*{int}(Z)\neq\emptyset$;

\item $\operatorname*{int}(Z)\subseteq\operatorname*{ri}%
(Z)=\operatorname*{iri}(Z)=\operatorname*{qri}(Z)\neq\emptyset$ whenever $V$
is finite-dimensional.
\end{itemize}

\noindent In the sequel we shall make use of Corollary \ref{c:co} that follows
as a consequence of Theorem \ref{t:cepaquart} below. Theorem \ref{t:cepaquart}
is in fact a restatement of Theorem 5 in \cite{Cepa19} and hence we omit its
proof.\footnote{Just note that in the mentioned article the definition of a
separating functional is reversed with respect to that used here (and that no
topological assumption is needed).}

\begin{theorem}
\label{t:cepaquart}Let $V$ be a TRVS and $C$ be a cone in $V$. Assume that $X
$ and $Y$ are nonempty subsets of $V$ and suppose $f\in V^{\ast}%
\backslash\{0\}$ separates $X$ and $Y$.

\begin{enumerate}
\item If $X$ is $C$-upward then $f$ is nonpositive on $C$.

\item If $X$ is $C$-downward then $f$ is nonpositive on $C$.\ \ \ \ \ \ \
\end{enumerate}
\end{theorem}

\begin{corollary}
\label{c:co} Let $V$ be a TRVS and $C$ be a cone in $V$. Assume that $X$ and
$Y$ are nonempty subsets of $V$ and suppose $f\in V^{\ast}\backslash\{0\}$
separates $X$ and $Y$.

\begin{enumerate}
\item If $X$ is $C$-upward then $f$ is nonpositive on $\operatorname*{co}%
(C\cup\{0\})$.

\item If $X$ is $C$-downward then $f$ is nonpositive on $\operatorname*{co}%
(C\cup\{0\})$.
\end{enumerate}
\end{corollary}

\begin{proof}
Suppose $X$ is $C$-upward (resp. $C$-downward). Pick an arbitrary
$k\in\operatorname*{co}(C\cup\{0\})$. Then there exist $n$ elements
$c_{1},\ldots,c_{n}$ in $C\cup\{0\}$ and $\alpha$ in $\mathbb{R}_{+}^{n}$ such
that $\alpha_{1}+\ldots+\alpha_{n}=1$ and $k=\alpha_{1}c_{1}+\ldots+\alpha
_{n}c_{n}$. The linearity of $f$ then implies that $f(k)=\alpha_{1}%
f(c_{1})+\ldots+\alpha_{n}f(c_{n})$. Part 1 (resp. Part 2) of Theorem
\ref{t:cepaquart} ensures that $f(c_{i})\leq0$ for all $i=1,\ldots,n$ and so
$f(k)\leq0$.
\end{proof}

\subsection{\label{NTI}Nonempty topological interior}

\begin{theorem}
\label{t:int}Let $V$ be a TRVS and $C$ be a cone in $V$. Assume that
$X_{1},\ldots,X_{m}$, $Y_{1},\ldots,Y_{n}$ are nonempty $C$-antichain-convex
subsets of $V$. Put
\[
X=X_{1}+\ldots+X_{m}\text{\ and\ }Y=Y_{1}+\ldots+Y_{n}\text{.}%
\]
Suppose $\operatorname*{int}(X)\neq0$ and $\operatorname*{int}(X)\cap
Y=\emptyset$.

\begin{enumerate}
\item If $X_{1}$ is $C$-upward then $X$ and $Y$ are separated.

\item If $X_{1}$ is $C$-downward then $X$ and $Y$ are separated.
\end{enumerate}
\end{theorem}

\begin{proof}
1. Suppose $X_{1}$ is $C$-upward. As $X_{1}$ is $C$-antichain-convex and
$C$-upward, part 1 of Corollary 1 in \cite{Cepa19} ensures that $X$ is convex.
The topological interior of a convex set is convex (see, e.g., Theorem 1.1.2
in \cite{Zali02}): thus $\operatorname*{co}(\operatorname*{int}%
(X))=\operatorname*{int}(X)$ and part 1 of Corollary \ref{Coro3} ensures that
$\operatorname*{int}(X)\cap\operatorname*{co}(Y)=\emptyset$. Consequently, the
Separation Theorem 14.2 in \cite{Kell63} guarantees the existence of $f$ in
$V^{\ast}\backslash\{0\}$ such that $\sup f[X]\leq\inf f[\operatorname*{co}%
(Y)]$. As $\emptyset\neq Y\subseteq\operatorname*{co}(Y)$ we have that $\inf
f[\operatorname*{co}(Y)]\leq\inf f[Y]$. We conclude that $\sup f[X]\leq\inf
f[Y]\text{.}$

2. A consequence of part 1 of Theorem \ref{t:int} and Remark \ref{K2}.
\end{proof}

\subsection{\label{CAC}Closed and compact sets\ \ \ \ \ \ \ \ }

\begin{theorem}
\label{KLOSE}Let $V$ be a LCS and $C$ be a cone in $V$. Assume that $X$ and
$Y$ are nonempty decomposably $C$-antichain-convex subsets of $V$. Suppose $X$
is closed, $\operatorname*{co}(Y)$ is compact and $X\cap Y=\emptyset$.

\begin{enumerate}
\item If $X$ is $C$-upward then $X$ and $Y$ are strictly separated.

\item If $X$ is $C$-downward then $X$ and $Y$ are strictly separated.
\end{enumerate}
\end{theorem}

\begin{proof}
1. Suppose $X$ is $C$-upward. As $X$ is decomposably $C$-antichain-convex and
$C$-upward, part 1 of Theorem \ref{STANGHI} ensures the convexity of $X$. So
$\operatorname*{co}(X)=X$ and part 1 of Theorem \ref{TEOCON} implies
$X\cap\operatorname*{co}(Y)=\emptyset$. Thus, by Corollary 14.4 in
\cite{Kell63} there exists $f\in V^{\ast}\backslash\{0\}$ such that $\sup
f[X]<\inf f[\operatorname*{co}(Y)]$. As $\emptyset\neq Y\subseteq
\operatorname*{co}(Y)$, we have that $\inf f[\operatorname*{co}(Y)]\leq\inf
f[Y]$. We conclude that $\sup f[X]<\inf f[Y]\text{.}$

2. A consequence of part 1 of Theorem \ref{KLOSE} and Remark \ref{K2}.
\end{proof}

\ \ \

Corollary \ref{t:klo} is a simple consequence for the finite-dimensional case.

\begin{corollary}
\label{t:klo}Let $C$ be a cone in $\mathbb{R}^{n}$. Assume that $X$ and $Y$
are nonempty closed decomposably $C$-antichain-convex subsets of
$\mathbb{R}^{n}$. Suppose $Y$ is bounded and $X\cap Y=\emptyset$.

\begin{enumerate}
\item If $X$ is $C$-upward then $X$ and $Y$ are strictly separated.

\item If $X$ is $C$-downward then $X$ and $Y$ are strictly separated.
\end{enumerate}
\end{corollary}

\begin{proof}
The convex hull of a compact subset of $\mathbb{R}^{n}$ is compact: see, e.g.,
Corollary 5.33 in \cite{Alip06}. Said this, the assertion follows directly
from Theorem \ref{KLOSE}.
\end{proof}

\subsection{\label{QQRRII}Quasi-relative interior}

In the following Lemma \ref{t:qriup} we use the characterization of a
quasi-relative interior as enunciated in Lemma 3.6 in \cite{Vanc19}. Other
characterizations have been proved in the literature (like, e.g., Proposition
2.16 in \cite{Borw92} or Theorem 2.3 in \cite{Flor13}): we use that in
\cite{Vanc19} for expositional convenience.

\begin{lemma}
\label{t:qriup} Let $V$ be a LCS and let $C$ be a cone in $V$. Assume that $X
$ is a convex subset of $V$.

\begin{enumerate}
\item If $X$ is $C$-upward, then $\operatorname*{qri}(X)$ is $C$-upward.

\item If $X$ is $C$-downward, then $\operatorname*{qri}(X)$ is $C$-downward.
\end{enumerate}
\end{lemma}

\begin{proof}
1. Suppose $X$ is $C$-upward. The proof is trivial if either
$\operatorname*{qri}(X)=\emptyset$ or $C=\emptyset$. Henceforth suppose
$\operatorname*{qri}(X)\neq\emptyset$ or $C\neq\emptyset$. Suppose
$z\in\operatorname*{qri}(X)$ and $c\in C$. Putting%
\begin{equation}
t=z+c\text{,}\label{lsep0}%
\end{equation}
we are done if we show that $t\in\operatorname*{qri}(X)$. By way of
contradiction, suppose $t\notin\operatorname*{qri}(X)$. Lemma 3.6 in
\cite{Vanc19} implies that $X$ and $\{t\}$ can be properly separated and so
there exists $f\in V^{\ast}\backslash\{0\}$ and $x_{0}\in X$ such that%
\begin{equation}
f(x)\leq f(t)\text{\ for all }x\in X\label{lsep1}%
\end{equation}
and%
\begin{equation}
f(x_{0})<f(t)\text{.}\label{lsep2}%
\end{equation}
Proper separation implies separation: part 1 of Theorem \ref{t:cepaquart} then
ensures that%
\begin{equation}
f(c)\leq0\text{.}\label{lsep4}%
\end{equation}
By the linearity of $f$, from (\ref{lsep1}), (\ref{lsep0}) and (\ref{lsep4})
we infer that
\begin{equation}
f(x)\leq f(z)\text{\ for all }x\in X\label{hsep}%
\end{equation}
and from (\ref{lsep2}), (\ref{lsep0}) and (\ref{lsep4}) we infer that%
\begin{equation}
f(x_{0})<f(z)\text{.}\label{speriamo3}%
\end{equation}
Then, inequalities in (\ref{hsep}) and (\ref{speriamo3}) imply that $X$ and
$\{z\}$ are properly separated and Lemma 3.6 in \cite{Vanc19} in turn implies
that $z\notin\operatorname*{qri}(X)$: a contradiction with the assumption that
$z\in\operatorname*{qri}(X)$.

2. A consequence of part 1 of Lemma \ref{t:qriup} and Remark \ref{K2}.
\end{proof}

\begin{theorem}
\label{t:qri} Let $V$ be a LCS and let $C$ be a cone in $V$. Assume that $X$
and $Y$ are two nonempty decomposably $C$-antichain-convex subsets of $V$.

\begin{enumerate}
\item Assume that $X$ is $C$-upward, that $\operatorname*{iri}(X)\neq
\emptyset$, that $\operatorname*{iri}(\operatorname*{co}(Y))\neq\emptyset$,
that $\operatorname*{qri}(X)\cap Y=\emptyset$ and that $X-Y$ is quasi-regular.
Then $X$ and $Y$ are properly separated.

\item Assume that $X$ is $C$-downward, that $\operatorname*{iri}%
(X)\neq\emptyset$, that $\operatorname*{iri}(\operatorname*{co}(Y))\neq
\emptyset$, that $\operatorname*{qri}(X)\cap Y=\emptyset$ and that $X-Y$ is
quasi-regular. Then $X$ and $Y$ are properly separated.
\end{enumerate}
\end{theorem}

\begin{proof}
1. As $X$ is decomposably $C$-antichain-convex and $C$-upward, part 1 of
Theorem \ref{STANGHI} ensures the convexity of $X$. Moreover, part 6 of Lemma
3 in \cite{Cepa19} and part 1 of Corollary 1 in \cite{Cepa19} ensure the
convexity of $X-Y$. So
\begin{equation}
X-Y=\operatorname*{co}(X-Y)=\operatorname*{co}(X)-\operatorname*{co}%
(Y)=X-\operatorname*{co}(Y)\text{.}\label{HBA}%
\end{equation}
The equalities in (\ref{HBA}) and the quasi-regularity of $X-Y$ entail the
quasi-regularity of $X-\operatorname*{co}(Y)$. The set $\operatorname*{qri}%
(X)$ is $C$-upward by part 1 of Lemma \ref{t:qriup} and is convex by
Proposition 2.11 in \cite{Borw92}.\ The assumption that $\operatorname*{qri}%
(X)\cap Y=\emptyset$ and part 1 of Theorem \ref{TEOCON} ensure that
\begin{equation}
\operatorname*{qri}(X)\cap\operatorname*{co}(Y)=\emptyset\text{.}%
\label{priemp}%
\end{equation}
As $\operatorname*{qri}(\operatorname*{co}(Y))\subseteq\operatorname*{co}(Y)$,
the equality in (\ref{priemp}) implies
\begin{equation}
\operatorname*{qri}(X)\cap\operatorname*{qri}(\operatorname*{co}%
(Y))=\emptyset\text{.}\label{HZP}%
\end{equation}
As $X-\operatorname*{co}(Y)$ is quasi-regular, the equality in (\ref{HZP}) and
the assumptions that $\operatorname*{iri}(X)\neq\emptyset$ and
$\operatorname*{iri}(\operatorname*{co}(Y))\neq\emptyset$ allow the
applicability of Theorem 5.3 in \cite{Vanc19}, which ensures the existence of
$f\in V^{\ast}\backslash\{0\} $ and $(x_{0},t_{0})\in X\times
\operatorname*{co}(Y)$ such that%
\begin{equation}
f(x)\leq f(t)\text{\ for all }(x,t)\in X\times\operatorname*{co}%
(Y)\label{qlsep1}%
\end{equation}
and%
\begin{equation}
f(x_{0})<f(t_{0})\text{.}\label{qlsep2}%
\end{equation}
As $Y\subseteq\operatorname*{co}(Y)$, from (\ref{qlsep1}) we infer that
\begin{equation}
f(x)\leq f(y)\text{\ for all }(x,y)\in X\times Y\text{.}\label{qfcco3}%
\end{equation}
As $t_{0}\in\operatorname*{co}(Y)$ and $Y$ is decomposably $C$%
-antichain-convex, part 2 of Theorem \ref{MS} ensures the existence of
$y_{0}\in Y$ such that $y_{0}\in t_{0}-\operatorname*{co}(C\cup\{0\})$. So
there exists $k\in\operatorname*{co}(C\cup\{0\})$ such that $t_{0}=y_{0}+k$.
The inequality in (\ref{qfcco3}) ensures that $f$ separates $X$ and $Y$: the
inequality in (\ref{qlsep2}), the linearity of $f$ and part 1 of Corollary
\ref{c:co} then imply
\begin{equation}
f(x_{0})<f(t_{0})=f(y_{0})+f(k)\leq f(y_{0})\text{.}\label{fps}%
\end{equation}
So the inequalities in (\ref{qfcco3}) and (\ref{fps}) imply the pope
separation of $X$ and $Y$.

2. A consequence of part 1 of Theorem \ref{t:qri} and Remark \ref{K2}.
\end{proof}

\

Corollary \ref{t:fint} is a simple consequence for the finite-dimensional case.

\begin{corollary}
\label{t:fint} Let $C$ be a cone in $\mathbb{R}^{n}$ and assume that $X$ and
$Y$ are two nonempty decomposably $C$-antichain-convex subsets of
$\mathbb{R}^{n}$.

\begin{enumerate}
\item If $X$ is $C$-upward and $\operatorname*{ri}(X)\cap Y=\emptyset$ then
$X$ and $Y$ are properly separated.

\item If $X$ is $C$-downward and $\operatorname*{ri}(X)\cap Y=\emptyset$ then
$X$ and $Y$ are properly separated.
\end{enumerate}
\end{corollary}

\begin{proof}
Corollary \ref{t:fint} is a direct consequence of Theorem \ref{t:qri}: just
note that $V=\mathbb{R}^{n}$ implies $\operatorname*{qri}%
(X)=\operatorname*{iri}(X)=\operatorname*{ri}(X)\neq\emptyset$,
$\operatorname*{iri}(\operatorname*{co}(Y))=\operatorname*{ri}%
(\operatorname*{co}(Y))\neq\emptyset$ and the quasi-regularity of $X-Y$.
\end{proof}

\section{\label{MAX}Maximals and maximizers}

We now investigate the structure of maximals of a $C$-antichain-convex
relation and that of the maximizers of a $C$-antichain-quasiconcave function
for some constrained optimization problems that can be frequently encountered
in economics. After proving some general facts about the $C$%
-antichain-convexity of optimal solutions in Sect. \ref{ACMM} and some
sufficient conditions that guarantee their incomparability (with respect to
the relation generated by the cone $C$) in Sect. \ref{AMM}, we show a result
on the convexity of the set of optimal solutions of non-convex optimization
problems in Sect. \ref{CMM}. Finally, in Sect. \ref{IUC}, we consider the
subtler problem of identifying conditions under which the set of maximals of a
relation is equal to that of its convexification.

\subsection{\label{PREFDE}Preliminary definitions}

In this Sect. \ref{PREFDE} we fix the definitions and notation used in the
remainder of paper. We refer to Sect. \ref{PRE} for all general definitions
concerning relations.

\paragraph*{Relations and functions\newline}

\noindent Let $V$ be a RVS and $R$ be a relation on a convex subset $X$ of $V
$. Then

\begin{itemize}
\item the \textbf{convexification of }$R$ is the relation
$R^{\operatorname*{co}}$ on $X$ defined by%
\begin{equation}
R^{\operatorname*{co}}(x)=\operatorname*{co}(R(x))\text{ for all }x\in
X\text{.}\label{COFIC}%
\end{equation}

\end{itemize}

\noindent Let $V$ be a RVS, let $C$ be a cone in $V$ and let $X$ be a
$C$-antichain-convex subset of $V$. Also, let $R$ be a relation on $X$ and
$u:X\rightarrow\mathbb{R}$ be a function. Like in \cite{Cepa19}, we say that:

\begin{itemize}
\item $R$ is $C$\textbf{-antichain-convex} iff $R(x)$ is $C$-antichain-convex
for all $x\in X$;

\item $u$ is $C$\textbf{-antichain-quasiconcave} iff $\{x\in X:u(x)\geq
\lambda\}$ is $C$-antichain-convex for all $\lambda\in\mathbb{R}$.
\end{itemize}

\begin{remark}
\label{JON}Let $V$ be a RVS, $C$ be a cone in $V$ and $u:X\rightarrow
\mathbb{R}$ be a function on a $C$-antichain-convex set $X\subseteq V$. The
$C$-antichain-quasiconcavity of $u$ implies the $C$-antichain-convexity of the
relation $R$ defined by (\ref{JENNY}) in Remark \ref{JOY}.
\end{remark}

\paragraph*{Local nonsatiation\newline}

\noindent Let $V$ be a RVS endowed with some topology and $X$ be a subset of
$V$. A relation $R$ on $X$ is \textbf{locally nonsatiated} iff
\[
x\in\operatorname*{cl}(\{y\in X:y\in R(x)\ \text{and}\ x\notin R(y)\})\text{
for all }x\in X\text{;}%
\]
a function $u:X\rightarrow\mathbb{R}$ is \textbf{locally nonsatiated} iff
\[
x\in\operatorname*{cl}(\{y\in X:u(y)>u(x)\})\text{ for all }x\in X\text{.}%
\]
Note that the operator $\operatorname*{cl}$ is meant with respect to the
topology of $V$.

\begin{remark}
\label{JAN}Let $V$ be a RVS endowed with some topology and $u:X\rightarrow
\mathbb{R}$ be a function on subset $X$ of $V$. Then the local nonsatiation of
$u$ is equivalent to the local nonsatiation of the relation $R$ defined by
(\ref{JENNY}) in Remark \ref{JOY}.
\end{remark}

\paragraph*{Positivity\newline}

\noindent Let $V$ be a RVS endowed with some topology and $X$ be a subset of
$V$. A continuous linear functional on $V$ that is positive on $X\backslash
\{0\}$ is called a \textbf{positive functional on }$X$. The set $P_{X}$
defined by
\[
P_{X}=\{f\in V^{\ast}:f(x)>0\text{ for all }x\in X\backslash\{0\}\}
\]
is the \textbf{set of all positive functionals on }$X$. For any $w\in
\mathbb{R}$ and $f\in V^{\ast}$, put%
\[
F_{f}^{w}=\{v\in V:f(v)\leq w\}
\]
and $B_{f,X}^{w}=\{x\in X:f(x)\leq w\}$. Clearly,%
\[
B_{f,X}^{w}=F_{f}^{w}\cap X\text{.}%
\]
Mathematically, is immaterial to name $F_{f}^{w}$ and $B_{f,X}^{w}$. However,
those sets have an economic meaning explained in Remark \ref{ECOIN} which
justifies our notation.

\begin{remark}
\label{ECOIN}The economic interpretation of $V$, $X$, $P_{X}$, $F_{f}^{w}$ and
$B_{f,X}^{w}$ is as follows: $V$ is the set of all commodity vectors and $X$
that of all consumption vectors; $f\in P_{X}$ is a price functional (or
better, $f(x)$ is the expenditure\footnote{\label{ATTENZ}When $V=\mathbb{R}%
^{n}$, the value of the price functional at the consumption $x\in X$ specifies
the expenditure $px$ given by the scalar product of $p\in\mathbb{R}^{n}$ and
$x$: the vector $p$ is called a price (and should not be confused with the
price functional which specifies the expenditure).} of a consumer for $x\in
X$); $F_{f}^{w}$ is the \textbf{set of all financially feasible commodity
vectors} and $B_{f,X}^{w}$ is that of all financially feasible consumption
vectors also called the \textbf{budget set}. The set $\mathcal{M}%
(R,B_{f,X}^{w})$ is the \textbf{demand} of a consumer with a preference
relation $R$ and wealth $w$, who chooses a consumption vector out of $X$
facing a price functional $f$.
\end{remark}

\subsection{\label{ACMM}Antichain-convexity of sets of maximals and of sets of
maximizers}

Theorem \ref{OSSY} and Corollary \ref{OSSYCO} show general facts on the
antichain-convexity of the set of maximals of a relation and of the set of
maximizers of a function.

\begin{theorem}
\label{OSSY}Let $V$ be a RVS and $C$ be a cone in $V$. Suppose $X$ and $S$ are
$C$-antichain-convex subsets of $V$ such that $S\subseteq X$. Besides suppose
$R$ is a total and $C$-antichain-convex relation on $X$. Then $\mathcal{M}%
(R,S)$ is $C$-antichain-convex.
\end{theorem}

\begin{proof}
Suppose $m^{\bullet}$ and $m^{\circ}$ are elements of $\mathcal{M}(R,S)$ such
that
\begin{equation}
m^{\bullet}-m^{\circ}\notin(C\cup-C)\text{.}\label{MILLY}%
\end{equation}
Pick an arbitrary $\lambda\in\lbrack0,1]$ and put $m=\lambda m^{\bullet
}+(1-\lambda)m^{\circ}$. We are done if we show that $m\in\mathcal{M}(R,S) $.
Pick an arbitrary $s\in S$. As $m^{\bullet}$ and $m^{\circ}$ are elements of
$\mathcal{M}(R,S)$, Lemma \ref{BASE} ensures that%
\begin{equation}
m^{\bullet}\in R(s)\text{ and }m^{\circ}\in R(s)\text{.}\label{MILLO}%
\end{equation}
As $R(s)$ is a $C$-antichain-convex subset of $S$, from (\ref{MILLY}) and
(\ref{MILLO}) we infer that $m\in R(s)$. This suffices to conclude that
$m\in\mathcal{M}(R,S)$.
\end{proof}

\begin{corollary}
\label{OSSYCO}Let $V$ be a RVS and $C$ be a cone in $V$. Suppose $X$ and $S$
are $C$-antichain-convex subsets of $V$ such that $S\subseteq X$. Besides
suppose $u:X\rightarrow\mathbb{R}$ is $C$-antichain-quasiconcave. Then
$\arg\max_{s\in S}u(s)$ is $C$-antichain-convex.
\end{corollary}

\begin{proof}
A consequence of Remark \ref{JON} and Theorem \ref{OSSY}.
\end{proof}

\subsection{\label{AMM}Antichains of maximals and of maximizers}

Theorem \ref{SAB} and Corollary \ref{SSAABB} show sufficient conditions for
the set of maximals of a relation and for that of maximizers of a function to
be $C$-antichains.

\begin{lemma}
\label{UGUA}Let $V$ be a RVS endowed with some topology. Suppose
$(w,f)\in\mathbb{R\times}V^{\ast}$. Then%
\[
z\in\operatorname*{bd}(F_{f}^{w})\Rightarrow f(z)=w\text{.}%
\]

\end{lemma}

\begin{proof}
Suppose $z\in\operatorname*{bd}(F_{f}^{w})$. Put $A=]-\infty,w[$ and
$B=]w,+\infty\lbrack$. As the sets $A$ and $B$ are $\mathbb{R}$-open and $f\in
V^{\ast}$, we have that the sets $f^{-1}[A]$ and $f^{-1}[B]$ are $V$-open.
Therefore, if $f(z)<w$ then there exists a $V$-neighborhood $N_{z}$ of $z$
such that $f(v)<w$ for all $v\in N_{z}$ while if $f(z)>$ $w$ then there exists
a $V$-neighborhood $N_{z}$ of $z$ such that $f(v)>w$ for all $v\in N_{z}$: a
contradiction with $z\in\operatorname*{bd}(F_{f}^{w})$. So $f(z)=w$.
\end{proof}

\begin{remark}
As it is clear from the proof of Lemma \ref{UGUA}, the linearity of $f$ can be
dispensed with in the statement of Lemma \ref{UGUA}. We assume it only for
expositional convenience.
\end{remark}

\begin{lemma}
\label{WW}Let $V$ be a RVS endowed with some topology and $X$ be a cone in $V
$. Suppose $(w,f)\in\mathbb{R\times}P_{X}$. Then $\operatorname*{bd}(F_{f}%
^{w}) $ is an $X$-antichain.
\end{lemma}

\begin{proof}
Suppose $x$ and $y$ are distinct elements in $\operatorname*{bd}(F_{f}^{w})$.
By way of contradiction, suppose
\begin{equation}
y-x\in X\cup-X\text{.}\label{HA1}%
\end{equation}
Lemma \ref{UGUA} implies%
\begin{equation}
f(y)=f(x)\text{.}\label{HA2}%
\end{equation}
As $y$ and $x$ are distinct, we have that
\begin{equation}
y-x\neq0\text{.}\label{HA3}%
\end{equation}
The assumption $f\in P_{X}$ implies the positivity of $f$ at all nonzero
vectors in $X$ and the negativity of $f$ at all nonzero vectors in $-X$.
Consequently, $f(y-x)\neq0$ by virtue of (\ref{HA1}) and (\ref{HA3}). The
linearity of $f$ in turn implies $f(y)\neq f(x)$: a contradiction with
(\ref{HA2}).\ \ \
\end{proof}

\begin{theorem}
\label{SAB}Let $V$ be a RVS endowed with some topology and $X$ be a cone in
$V$. Suppose $(w,f)\in\mathbb{R\times}P_{X}$ and $R$ is a locally nonsatiated
total relation on $X$. Then $\mathcal{M}(R,B_{f,X}^{w})$ is an $X $-antichain
included in $\operatorname*{bd}(F_{f}^{w})$.
\end{theorem}

\begin{proof}
By virtue of Lemma \ref{WW}, we are done if we show that $\mathcal{M}%
(R,B_{f,X}^{w})\subseteq\operatorname*{bd}(F_{f}^{w})$. By contradiction,
suppose
\begin{equation}
m\in\mathcal{M}(R,B_{f,X}^{w})\label{Reno0}%
\end{equation}
and
\begin{equation}
m\notin\operatorname*{bd}(F_{f}^{w}).\label{Reno1}%
\end{equation}
As $m\in B_{f,X}^{w}\subseteq F_{f}^{w}$, from (\ref{Reno1}) we infer that
$m\in\operatorname*{int}(F_{f}^{w})$. The previous membership implies the
existence of a $V$-neighborhood $N_{m}$ of $m$ such that $N_{m}\subseteq
F_{f}^{w}$ and the local nonsatiation of $R$ in turn implies the existence of
$x\in N_{m}\cap X$ such that
\begin{equation}
x\in R(m)\text{ and }m\notin R(x)\text{.}\label{Reno2}%
\end{equation}
As $x\in N_{m}\cap X\subseteq F_{f}^{w}\cap X$, noting that $X\cap F_{f}%
^{w}=B_{f,X}^{w}$ we conclude that $x\in B_{f,X}^{w}$: a contradiction with
(\ref{Reno0}) and (\ref{Reno2}).
\end{proof}

\begin{corollary}
\label{SSAABB}Let $V$ be a RVS endowed with some topology and $X$ be a cone in
$V$. Suppose $(w,f)\in\mathbb{R\times}P_{X}$ and $R$ is a locally nonsatiated
total relation on $X$. If $C\subseteq X$ is a cone in $V$ then $\mathcal{M}%
(R,B_{f,X}^{w})$ is a $C$-antichain included in $\operatorname*{bd}(F_{f}%
^{w})$.
\end{corollary}

\begin{proof}
Theorem \ref{SAB} ensures that $\mathcal{M}(R,B_{f,X}^{w})$ is an
$X$-antichain. A fortiori, $\mathcal{M}(R,B_{f,X}^{w})$ is a $C$-antichain.
\end{proof}

\subsection{\label{CMM}Convexity of sets of maximals and of sets of
maximizers}

Using the results obtained in Sect. \ref{ACMM} and \ref{AMM}, Theorems
\ref{PAZ} and \ref{PIZZ} derive sufficient conditions for the convexity of the
set of maximals of a relation and for that of maximizers of a function.

\begin{theorem}
\label{PAZ}Let $V$ be a RVS endowed with some topology and $X$ be a convex
cone in $V$. Suppose $(w,f)\in\mathbb{R\times}P_{X}$ and $R$ is a locally
nonsatiated total relation on $X$. If $C\subseteq X$ is a cone in $V$ and $R$
is $C$-antichain-convex then $\mathcal{M}(R,B_{f,X}^{w})$ is a convex
$C$-antichain included in $\operatorname*{bd}(F_{f}^{w})$.
\end{theorem}

\begin{proof}
The set $F_{f}^{w}$ is readily seen to be convex. So also $B_{f,X}^{w}$ is
convex (as it is the intersection of the convex sets $X$ and $F_{f}^{w}$). A
fortiori, $B_{f,X}^{w}$ is $C$-antichain-convex by Proposition 2 in
\cite{Cepa19}. Said this, Theorem \ref{PAZ} is a consequence of Theorem
\ref{OSSY} and Corollary \ref{SSAABB}.\ \
\end{proof}

\begin{theorem}
\label{PIZZ}Let $V$ be a RVS endowed with some topology and $X$ be a convex
cone in $V$. Suppose $(w,f)\in\mathbb{R\times}P_{X}$ and $u:X\rightarrow
\mathbb{R}$ is a locally nonsatiated function. If $C\subseteq X$ is a cone in
$V$ and $u$ is $C$-antichain-quasiconcave then $\arg\max_{x\in B_{f,X}^{w}%
}u(x)$ is a convex $C$-antichain included in $\operatorname*{bd}(F_{f}^{w})$.
\end{theorem}

\begin{proof}
A consequence of Theorem \ref{PAZ} and Remarks \ref{JON} and \ref{JAN}.
\end{proof}

\ \

Theorems \ref{PAZ} and \ref{PIZZ} are of importance to economics as they show
that, in many optimization problems which can be encountered therein, the
usual convexity assumptions are not necessary and can be relaxed. More
concretely---and keeping in mind footnote \ref{ATTENZ}---given a real-valued
locally nonsatiated (utility) function $u$ on $X=\mathbb{R}_{+}^{n}$, Theorem
\ref{PIZZ} guarantees that the convexity of the Walrasian demand
correspondence $x:\mathbb{R}_{++}^{n}\times\mathbb{R}_{+}\rightarrow2^{X}$
defined by
\[
x(p,w)=\underset{s\in\{y\in X:py\leq w\}}{\arg\max}u(s)\text{,}%
\]
(with $py$ representing the expenditure at $y$ given a price $p\in
\mathbb{R}_{++}^{n}$ and $w$ representing the consumer's wealth) obtains when
$u$ is $\mathbb{R}_{+}^{n}$-antichain-quasiconcave: the stronger assumption of
quasiconcavity is not necessary. So, for instance, Theorem \ref{PIZZ}
guarantees that the real-valued (locally nonsatiated) function $u$ on
$X=\mathbb{R}_{+}^{2}$ defined by%
\[
u(x_{1},x_{2})=\frac{x_{1}x_{2}}{x_{1}+1}-5x_{1}+x_{2}%
\]
(which is $\mathbb{R}_{+}^{2}$-antichain-quasiconcave, albeit not
quasiconcave, by Example 9 and Remark 8 in \cite{Cepa19}) generates a
convex-valued demand correspondence.

\subsection{\label{IUC}On invariance under convexification}

The previous theorems are not sufficient to guarantee the equivalence of the
set of maximals of a relation and that of its convexification. Theorem
\ref{GRINTA} shows that a strengthening of the conditions posited in Theorem
\ref{PAZ} allows to obtain the desired equivalence result. The most important
additional assumption that we impose is the existence of a maximal. In
finite-dimensional spaces, the compactness of the (budget) set $B_{f,X}^{w}$
obtains under reasonable economic assumptions and the existence of maximals is
not a real issue; in infinite-dimensional spaces, however, the compactness of
$B_{f,X}^{w}$ does not generally hold and maximals need not exist. The
difficult issue of the existence of maximals in infinite-dimensional spaces
has been investigated and has received some answers: see \cite{Poly08} and the
literature cited therein.

\begin{theorem}
\label{GRINTA}Let $V$ be a RVS endowed with some topology and $X$ be a convex
cone in $V$ containing the zero vector. Suppose $(w,f)\in\mathbb{R\times}%
P_{X}$ and $R$ is a locally nonsatiated total preorder relation on $X$. If
$C\subseteq X$ is a convex cone in $V$ and $R$ is $C$-antichain-convex then
\[
\mathcal{M}(R,B_{f,X}^{w})\neq\emptyset\Rightarrow\mathcal{M}(R,B_{f,X}%
^{w})=\mathcal{M}(R^{\operatorname*{co}},B_{f,X}^{w})\text{.}%
\]

\end{theorem}

\begin{proof}
Recall that $C\cup\{0\}$ is a convex cone in $V$ by virtue of part 1 of Lemma
\ref{ERRA}. Part 2 of Lemma 3 in \cite{Cepa19} implies that $R(v)$ is
$C\cup\{0\}$-antichain-convex for all $v\in X$. Recalled these facts and that
$0\in X$ by assumption, henceforth suppose without loss of generality that
$0\in C$. The previous membership and the convexity of $C$ imply
$C=\operatorname*{co}(C\cup\{0\})$. As $R\subseteq R^{\operatorname*{co}}$, we
have that $R^{\operatorname*{co}}$ is total. Noted this, we can apply Lemma
\ref{PEL} concluding that%
\begin{equation}
\mathcal{M}(R,B_{f,X}^{w})\subseteq\mathcal{M}(R^{\operatorname*{co}}%
,B_{f,X}^{w})\text{.}\label{GINO}%
\end{equation}
Suppose $\mathcal{M}(R,B_{f,X}^{w})\neq\emptyset$ and pick an arbitrary%
\begin{equation}
y\in\mathcal{M}(R,B_{f,X}^{w})\text{.}\label{FCA}%
\end{equation}
We conclude the proof showing that the converse of the inclusion in
(\ref{GINO}) is true. So, suppose that%
\begin{equation}
m\in\mathcal{M}(R^{\operatorname*{co}},B_{f,X}^{w})\label{Gin1}%
\end{equation}
and, by way of contradiction, that $m\notin\mathcal{M}(R,B_{f,X}^{w})$: then
there exists $y^{\ast}\in B_{f,X}^{w}$ such that $y^{\ast}\in R(m)$ and
\begin{equation}
m\notin R(y^{\ast})\text{.}\label{REST}%
\end{equation}
By part 2 of Lemma \ref{AIUT}, from (\ref{FCA}) and (\ref{REST}) we infer that
$y\in R(m)$ and
\begin{equation}
m\notin R(y)\text{.}\label{Gin2}%
\end{equation}
By Lemma \ref{BASE}, the totality of $R^{\operatorname*{co}}$ and the
membership in (\ref{Gin1}) imply $m\in R^{\operatorname*{co}}(y)$: part 2 of
Lemma \ref{HM} in turn implies the existence of an element%
\begin{equation}
x\in R(y)\label{Gin4}%
\end{equation}
such that $x\in m-C$. From (\ref{Gin2}) and (\ref{Gin4}) we conclude that
$x\neq m$. As $x\in m-C$ and $x\neq m$, there exists $c^{\ast}$ such that%
\begin{equation}
c^{\ast}\in C\backslash\{0\}\text{ and }x=m-c^{\ast}\text{.}\label{DOG}%
\end{equation}
By part 1 of Lemma \ref{AIUT}, from (\ref{FCA}) and (\ref{Gin4}) we infer that
$x\in\mathcal{M}(R,B_{f,X}^{w})$: Corollary \ref{SSAABB} and Lemma \ref{UGUA}
then imply%
\begin{equation}
f(x)=w\text{.}\label{PNZ}%
\end{equation}
The membership in (\ref{Gin1}) implies $m\in B_{f,X}^{w}$ and hence%
\begin{equation}
f(m)\leq w\text{.}\label{PNZ1}%
\end{equation}
As $f\in P_{X}$ and $C\subseteq X$, from (\ref{DOG}) we infer that $f(c^{\ast
})>0$ and $f(x)=f(m)-f(c^{\ast})$: then (\ref{PNZ1}) implies $f(x)<w$ in
contradiction with (\ref{PNZ}).
\end{proof}

\newpage

\bibliographystyle{abbrv}
\bibliography{biblio_nota}

\end{document}